\documentclass{amsart}

\usepackage[T1]{fontenc}
\usepackage[latin9]{inputenc}
\usepackage{bm}
\newtheorem{thm}{Theorem}

\newtheorem{lemma}{Lemma}
\newtheorem{cor}{Corollary}
\def\Tr{\mbox{Tr }}
\title{A note on extremal decompositions of covariances}
\author{Zolt\'an L\'eka}
\address{Alfr\'ed R\'enyi Institute of Mathematics \\ 1053 Budapest \\  Re\'altanoda u. 13-15 \\ Hungary }
\email{leka.zoltan@renyi.mta.hu}


\thanks{This study was partially supported by Hungarian NSRF (OTKA) grant
no. K104206 }

\begin{document}

\begin{abstract}
    We shall present an elementary approach to extremal decompositions of (quantum) covariance matrices determined by densities. We give a new proof on former results and provide a
    sharp estimate of the ranks of the densities that appear in the decomposition theorem.
\end{abstract}

\maketitle
\section{Introduction}

 Let $D \in M_n(\mathbb{C})$ denote an $n \times n$ (complex) density matrix (i.e. $D \geq 0$ and $\Tr D = 1$), and let $X_i$ ($1 \leq i \leq k$) stand for self-adjoint matrices in $M_n(\mathbb{C}).$ Then the 
 non-commutative covariance matrix is defined by
 $$ \mbox{Var}_D(\mathbf{X})_{ij} :=  \mbox{Tr } DX_iX_j - \left(\mbox{Tr } DX_i\right) \left(\mbox{Tr } DX_j\right) \quad 1 \leq i,j \leq k,$$
 where $\mathbf{X}$ stands for the tuple $(X_1, \hdots, X_k),$ see \cite[p. 13]{P}.
 We note that there are more general versions of variances and covariance matrices. For instance, in \cite{B}, \cite{BD} R. Bhatia and C. Davis introduced them by means of completely positive maps and applied the concept for improving non-commutative
 Schwarz inequalities.
 
 Covariances naturally appear in quantum information theory as well and it seems that there is a recent interest in order to understand their extremal properties \cite{PD}, \cite{PD2}.
 More precisely, in \cite{PD} D. Petz and G. T\'oth proved that any density matrix $D$ can be written as the convex combination of projections $\{P_l\},$ i.e. $D = \sum_l \lambda_l P_l,$ such that $$ \mbox{Var}_D(X) =  \sum_l \lambda_l \mbox{Var}_{P_l}(X)$$
 holds, where $X$ denotes a fixed Hermitian. It is worth it to mention here that quite recently S. Yu pointed out some extremal aspects of the variances which yields a descriptions
 of the quantum Fisher information in terms of variances (for the details, see \cite{Y}).
 
 In this short note we study analogous questions in the multivariable case.  
 Actually, we are interested in the following problem: let us find densities $D_l \in M_n(\mathbb{C})$ such that 
   $$ D = \sum_l \lambda_lD_l  \quad \mbox{and} \quad \mbox{Var}_D(\mathbf{X}) = \sum_l \lambda_l \mbox{Var}_{D_l}(\mathbf{X}),$$   where $\sum_l \lambda_l = 1$ and $0 < \lambda_l < 1.$
 Let us call a density $D$ {\bf extreme  with respect to} $\mathbf{X} = (X_1, \hdots, X_k)$ if it admits only the trivial decomposition (i.e. $D_l = D$ for every $l$).
 It was proved in the cases $k = 1$ and $k = 2$ that the extreme densities are rank-one projections \cite{LP}, \cite{PD}. Furthermore, the number of projections used, i.e. the length of
 the decomposition, is polynomial in rank $D$ (see \cite{LP}).
 
 The aim of this note is to present a simple approach to the extremal problem above and to look at the question 
 from the theory of extreme correlation matrices (see \cite{CV},\cite{GPW} and \cite{CT}). 
 In this context we shall give a new proof to the decomposition theorems appeared in \cite{LP}, \cite{PD}, \cite{PD2} and we present a sharp rank-estimate of the extreme densities.

 \section{Results and examples}
 
 First we collect some basic properties of the covariance matrix $\mbox{Var}_D(\mathbf{X}).$ We note that the matrix does not change by (real) scalar perturbations of the tuple $(X_1, \hdots, X_k).$ In fact, an elementary calculation on the entries gives that 
   $$(1) \qquad  \mbox{Var}_D(\mathbf{X}) = \mbox{Var}_D(X_1 - \lambda_1I, \hdots, X_k - \lambda_kI),$$ where $\lambda_i \in \mathbb{R}$ for every $i.$ Moreover, one can readily check that $\mbox{Var}_D(\mathbf{X})$ is positive.
   For the sake of completeness, here is a simple proof. 
   
   \begin{lemma}
       $ {\rm Var}_D(\mathbf{X}) \geq 0. $
   \end{lemma}
 
  \begin{proof}
    By (1), without loss of generality, one can assume that Tr $DX_i = 0$ holds for every $1 \leq i \leq k.$ The density $D$ defines a semi--inner
    product $ \langle A,B \rangle_D := \Tr DA^*B$ on $M_n(\mathbb{C}).$ Since ${\rm Var}_D(\mathbf{X})_{ij} = \langle X_i, X_j \rangle_D,$ for any $y = (y_1, \hdots, y_k) \in \mathbb{C}^k,$ we get that
     $$ y {\rm Var}_D(\mathbf{X}) y^* = \langle \: \sum_i y_i X_i,\sum_i y_i X_i \rangle_D  \geq 0$$ and the proof is done. 
  \end{proof}
 
  Next we show that the covariance is a concave function on the set of the density matrices. 
 
 \begin{lemma}  Let $D = \sum_l \lambda_l D_l$ be a finite sum of densities $D_l  \in M_n(\mathbb{C})$ such that $\sum_l \lambda_l = 1$ and $0 \leq \lambda_l \leq 1.$ Then     
     $$ {\rm Var}_D(\mathbf{X}) \geq \sum_l \lambda_l {\rm Var}_{D_l}(\mathbf{X}). $$
 \end{lemma}    
 
 \begin{proof}
  Choose $0 < \lambda < 1.$ If $D = \lambda D_1 + (1-\lambda)D_2,$ a straightforward calculation gives that
      $${\rm Var}_D(\mathbf{X}) - (\lambda{\rm Var}_{D_1}(\mathbf{X}) + (1-\lambda) {\rm Var}_{D_2}(\mathbf{X})) = \lambda(1-\lambda)[x_{ij}]_{1 \leq i,j \leq k},$$
     where $x_{ij} = {\rm Tr} \: (D_1-D_2)X_i {\rm Tr} \: (D_1-D_2)X_j.$ Therefore $[x_{ij}]_{1 \leq i,j \leq k} = XX^* \geq 0$ holds with
     $$ X =  \left[\begin{matrix} {\rm Tr} \: (D_1-D_2)X_1 & 0 & \hdots & 0 \cr
                                                              \vdots & \vdots &  & \vdots \cr 
                                                              {\rm Tr} \: (D_1-D_2)X_k & 0 & \hdots & 0 \end{matrix}\right] \: \in \: M_k(\mathbb{C}),$$ 
                                                              and the lemma readily follows. 
  \end{proof}

   The scalar perturbation property $\mbox{Var}_D(\mathbf{X}) = \mbox{Var}_D(\mathbf{X} - {\bf \lambda})$ guarantees that 
   it is enough to solve the extremal problem when $\mbox{Tr } DX_i = 0$ comes for every $1 \leq i \leq k.$ Then the nonlinear part of the covariance vanishes, thus we can simply transform our problem into a 
   geometrical one: let $X_i \in M_n(\mathbb{C})$ ($1 \leq i \leq k$) be self-adjoints and define the set
    \begin{eqnarray*}
     \begin{split}
      \mathcal{D}(\mathbf{X}) := \{ D \colon D \in M_n(\mathbb{C}) &\mbox{ is density and } \\
     &\mbox{Tr } DX_i = 0  \mbox{ for every } 1 \leq i \leq k \}.      
     \end{split}
    \end{eqnarray*} 
   Clearly, $\mathcal{D}(\mathbf{X})$ is a convex, compact set. From the Krein--Milman theorem, $\mathcal{D}(\mathbf{X})$ is the convex hull of its extreme points. Precisely, these extreme points are the extreme densities we are looking
   for in the decomposition of Var$_D(\mathbf{X}).$ 
   
\bigskip
 
 Notice that there is no restriction if we assume that $X_1, \hdots, X_k$ are linearly independent over $\mathbb{R}.$  
 Hence from here on we shall use this assumption on $X_i$-s.

 When $k \geq 3,$ one can see that it is no longer true that the extreme points of $\mathcal{D}(\mathbf{X})$ are rank-one projections. 
 In fact, look at the following simple example in $M_2(\mathbb{C})$ with $k = 3.$ 
 
 \bigskip
 
  \noindent {\bf Example 1.} Recall that the Pauli matrices are given by 
 $$
\sigma_x = \left[\begin{matrix}  0   & 1
\cr 1 & 0 \end{matrix}\right] \qquad  \sigma_y = \left[\begin{matrix} 0   & {\rm -i}
\cr {\rm i} & 0 \end{matrix}\right] \qquad \sigma_z = \left[\begin{matrix} 1   & 0
\cr 0 & -1 \end{matrix}\right].
$$ 
Any $2 \times 2$ Hermitian $Z$ with Tr $Z = 1$ can be expressed in the form
$$ 
Z = {1 \over 2}(I_2 + x\sigma_x + y\sigma_y + z \sigma_z),
$$ where $x,y$ and $z \in \mathbb{R}.$ Then the points of the Bloch sphere, i.e. $x^2 + y^2 + z^2 = 1,$ correspond to the rank-one projections. It is standard that the self-adjoints of trace $1,$ 
which are orthogonal to a fixed $Z,$ form an affine $2$-dimensional subspace of $\mathbb{R}^3.$
Hence one can find $X_1, X_2$ and $X_3$ so that the only density $D$ that satisfies $\mbox{Tr } DX_i = 0$ $(1 \leq i \leq 3)$ is inside the Bloch ball.
Then $\mathcal{D}(\mathbf{X})= \{ D \}$ and  $D$ is a density of rank $2.$
  
 \bigskip 
  
    We shall present a simple characterization of extreme densities or the extreme points of $\mathcal{D}(\mathbf{X}).$
    We recall that for any positive operators $D$ and $C,$ $D - \varepsilon C$ is positive for some $\varepsilon > 0$ if and
    only if $\mbox{ran } C \leq \mbox{ran } D$ holds. Then we can prove 
    
    \begin{lemma}
     The following statements are equivalent:
     \begin{itemize}
      \item [(i)]  $D$ is an extreme point of $\mathcal{D}(\mathbf{X}),$
      \item [(ii)] if $C \in \mathcal{D}(\mathbf{X})$ such that ${\rm ran} \: C \leq {\rm ran } \: D$ then $C = D.$  
     \end{itemize}
    \end{lemma}

    \begin{proof}
      Let us assume that ${\rm ran} \: C \leq {\rm ran } \: D$ and $D \neq C \in \mathcal{D}(\mathbf{X}).$ Then 
       $$ (1-\varepsilon) \left( {1 \over 1 - \varepsilon} (D - \varepsilon C) \right) + \varepsilon C = D,$$
       where $0 < \varepsilon < 1,$ hence $D$ cannot be an extreme point of $\mathcal{D}(\mathbf{X}).$ 
       
       Conversely, if $D$ is not extreme then $D = {1 \over 2} D_1 + {1 \over 2} D_2$ which implies that $\mbox{ran } D - {1 \over 2} D_1 \leq \mbox{ran } D,$  since  $D - {1 \over 2} D_1 $ is positive.
    \end{proof}

   To produce a description of ext $\mathcal{D}(\mathbf{X})$ which is more effective for our purposes, we need some basic facts about correlation matrices. We recall that a positive semidefinite matrix is a correlation matrix
   if its diagonal entries are $1$-s. Correlation matrices form a convex, compact set in $M_n(\mathbb{C}).$ Its extreme points, or extreme correlation
   matrices, were described by several authors, see e.g. \cite{GPW}, \cite{CT}. It is well-known that an $n \times n$ extreme correlation matrix has rank at most $\sqrt{n}$ (see e.g. \cite{CV}).
   Later we shall present an estimate of the rank of extreme densities matrices (with respect to tuples).
   
   The perturbation method used by C.-K. Li and B.-S. Tam is relevant for us. Let us say that 
   a nonzero Hermitian $S \in M_n(\mathbb{C})$ is a {\bf perturbation} of $D$ if there exists an  $\varepsilon > 0$ such that $D \pm \varepsilon S$ are density matrices as well. Then $D$ is an extreme 
   density with respect to $X_1, \hdots, X_k$ if and only if there does not exist perturbation $S$ of $D$ such that Tr $S = 0$ and Tr $SX_i = 0$ for every 
   $1 \leq i \leq k.$ In fact, if $D$ is not extreme, one can find $D_1$ and $D_2$ densities such that $D = {1 \over 2} D_1 + {1\over 2} D_2$ and $\Tr D_jX_i = 0.$ It follows that $S = D_1 - D_2$ is a perturbation of $D.$ The converse statement is trivial. 
   
   From here on let $H_n(\mathbb{C})$ denote the real Hilbert space of $n \times n$ complex Hermitian matrices with the usual 
    inner product $\langle A, B \rangle = \Tr AB.$
   One can easily conclude that an extreme density $D$ (with respect to $\mathbf{X}$) must be singular if $n^2 > k+1.$ Actually, 
   the last inequality guarantees the existence of a Hermitian perturbation $S$ which satisfies the orthogonality constraints; i.e. $S$ is orthogonal to $X_i$-s and $I.$ Moreover, the continuity of the spectra here gives that any small perturbation $D \pm \varepsilon S$ is positive if $D$ is invertible. 
   
     Let $\sigma(A)$ denote the spectrum of any $A \in M_n(\mathbb{C}).$
   Suppose that the matrix $D$ is of rank $r$. Then there does exist an $Y \in M_{n \times r}(\mathbb{C})$ and $R \in H_r(\mathbb{C})$ such that
   $D = YRY^*.$ Now one can prove the following lemma which is analogous to \cite[Theorem 1. (a)]{CT}.
   
      \begin{lemma}
      Let $D = YRY^* \in  \mathcal{D}(\mathbf{X})$ be a density of rank $r.$ 
      Then  $S$ is a perturbation of $D$ if and only if ${\rm Tr} \: S = 0$ and $S = YQY^*$ where $Q \in H_r(\mathbb{C}).$ 
   \end{lemma}

   \begin{proof}
      First, assume that $S = YQY^*.$ Then  $S$ is nonzero if and only if $Q \neq 0.$ Indeed, we have $\mbox{rank } S = \mbox{rank } Q$ because $Y$ has full column rank $r.$ 
      Since $D = YRY^*$ is positive, we obtain that $R$ is positive and invertible. From $0 \notin \sigma(R),$ there does exist an $\varepsilon > 0,$ such that  $D \pm \varepsilon S = Y(R\pm \varepsilon Q)Y^* $ are positive. 
      Obviously, we get that $S$ is a perturbation. 
      
      Conversely, let us assume that $S$ is perturbation of $D.$ Clearly, Tr $S = 0$ must hold. Expand $Y$ with a matrix $Z \in M_{n \times (n-r)}(\mathbb{C})$ such that 
      $V = (Y|Z)$ is invertible and $V(R \oplus 0_{n-r})V^* = D$ hold.  Next, let us write $V^{-1}S(V^*)^{-1}$ into blocks that corresponds to the block form of $R \oplus 0_{n-r}.$
      Since $V^{-1}(D \pm \varepsilon S)(V^{-1})^*$ are positive for some $\varepsilon > 0,$ it follows that $S = V(Q \oplus 0_{n-r})V^* $ must hold for some $ Q \in H_r(\mathbb{C}).$ \\
    \end{proof} 
  
   After this lemma here is our main result which reflects some similarity with the characterization theorem of extreme correlations, see \cite[Theorem 1]{CT}.  
    
     \begin{thm}
       Let $X_i \in H_n(\mathbb{C}),$ $1 \leq i \leq k,$ and $D = YRY^* \in  \mathcal{D}(\mathbf{X})$ be a density of rank $r,$
       where $Y \in M_{n \times r}(\mathbb{C}).$ The followings are equivalent:
       \begin{itemize}
        \item[(i)] $D$ is an extreme point of $\mathcal{D}(\mathbf{X}),$
        \item[(ii)] $ {\rm span} \:  \{ Y^*X_1Y, \hdots, Y^*X_kY, Y^*Y  \} = H_r(\mathbb{C}),$      
        \item[(iii)] $  \{ DX_1D, \hdots, DX_kD, D^2  \} \mbox{ has (real) rank } r^2.$
       \end{itemize}
       Moreover, if $D = YY^*$ then the above statements are equivalent to   
       \begin{itemize}
        \item[(iv)] $r^{-1}I_r$ is an extreme density with respect to $Y^*\mathbf{X}Y;$  
         that is, $$\mathcal{D}(Y^*\mathbf{X}Y) = \{r^{-1} I_r\}.$$
       \end{itemize}
      \end{thm}  
     \begin{proof} 
      \noindent (i) $\Leftrightarrow$ (ii) From Lemma 4, $D$ is extreme if and only if there does not exist $0 \neq YQY^*$ such that Tr $YQY^*Y_i = \Tr Q(Y^*X_iY) = 0$ and Tr $YQY^* = \Tr  Q(Y^*Y) = 0.$ We notice that $Q = 0$ if and only
      if the linear span of $Y^*X_1Y, \hdots,$ $Y^*X_kY$ and $Y^*Y$ is the full space $H_r(\mathbb{C}).$ \\
     \noindent (iii) $\Leftrightarrow$  (ii) Let us choose the decomposition $D = YY^*;$ that is, $R = I_r.$ Note that the self-adjoint $Y^*Y \in M_r(\mathbb{C})$ is invertible. In fact, $\sigma(YY^*) \cup \{0\} = \sigma(Y^*Y) \cup \{0\}$ holds, thus
      $\sigma(Y^*Y)$ equals to the set of positive eigenvalues of $D$ (with multiplicities). This implies that 
      $\sum_{i = 0}^k  \alpha_i Y^*X_iY = 0$ if and only if $\sum_{i = 0}^k \alpha_i  YY^*X_iYY^* = 0$ $(\alpha_i \in \mathbb{R}, \: X_{0} = I_n),$ so the systems $\{ Y^*X_1Y, \hdots, Y^*X_kY, Y^*Y  \}$ and $\{ DX_1D, \\ \hdots, DX_kD, D^2  \}$ have the same rank. \\
      \noindent (i) $\Rightarrow$  (iv) Since $D$ is an extreme point, we get from (ii) that $\{Y^*X_1Y, \\ \ldots, Y^*X_kY\}$ has rank at least $r^2-1.$
      However, $I_r$ is not in the linear span of the above system because it is orthogonal to every matrix $Y^*X_iY.$ Adjusting $r^{-1}I_r$ to $Y^*\mathbf{X}Y$, we get a full rank system
      of $H_r(\mathbb{C}).$ Hence by (iii) we conclude that $r^{-1}I_r$ is an extreme point of $\mathcal{D}(Y^*\mathbf{X}Y).$ \\
      \noindent (iv) $\Rightarrow$  (i) If $r^{-1}I_r$ is an extreme point, it has no perturbation $S$ which is orthogonal to every $Y^*X_iY.$ Thus it follows that
      $I_r, Y^*X_1Y, \ldots, Y^*X_kY$ must span $H_r(\mathbb{C});$ that is, $\mathcal{D}(Y^*\mathbf{X}Y) = \{r^{-1}I_r\}.$ Note that $Y^*Y, Y^*X_1Y, \\ \ldots,$ $Y^*X_kY$ span $H_r(\mathbb{C})$ as well becase $\Tr Y^*Y = \Tr D = 1$ and $Y^*X_iY$-s are traceless. Thus
      (ii) implies that $D$ is an extreme point. 
     \end{proof} 
    
   The theorem gives a straightforward estimate of the rank of extreme densities. 
   
   \begin{cor}
    Let $D \in M_n(\mathbb{C})$ be an extreme density with respect to $X_1, \hdots, X_k \in H_n(\mathbb{C}).$ 
    Then 
     $$ {\rm rank} \: D \leq \sqrt{k+1}. $$
   \end{cor}

  The Krein--Milman theorem implies that  $\mbox{Var}_D(\mathbf{X})$ can be written as the convex sum of covariances determined by
  densities of rank at most $\sqrt{k+1}.$
  Moreover, one can easily deduce the following result which first appeared in \cite{LP}, \cite{PD2} and \cite[Theorem]{PD}.
   
   \begin{cor}
    Let $D \in M_n(\mathbb{C})$ denote a density matrix. In the case of $k = 1$ and $k = 2,$ there exist projections $P_1, \hdots, P_m$ such that 
         $$D = \sum_{l=1}^m \lambda_l P_l  \quad \mbox{and} \quad {\rm Var}_D(\mathbf{X}) = \sum_{l=1}^m \lambda_l {\rm Var}_{P_l}(\mathbf{X})$$ hold, 
        where $\sum_{l=1}^m \lambda_l = 1$ and $0 \leq \lambda_l \leq 1.$
   \end{cor}

   In the case of $k \geq 3$, one might expect that the covariance matrix still can  be decomposed by means of projections if $n$ is large enough. 
   However, this is not necessarily true. The next example shows that the estimate of Corollary 1 is sharp if $n$ is large enough.
   
   \bigskip
   
   \noindent {\bf Example 2.} Let $n = \lfloor \sqrt{k+1} \rfloor.$ The special unitary group $SU(n)$ has dimension $n^2-1,$ so let $\lambda_i$ $(1 \leq i \leq n^2-1)$ denote a collection of its traceless, Hermitian infinitesimal generators. 
   One can also assume that $\Tr \lambda_i \lambda_j = 0$ holds for every $i \neq j$ (for the generalized Gell--Mann matrices, see e.g. \cite{SZ}). Then the matrices $\{I_n, \lambda_1, \hdots,  \lambda_{n^2-1}\}$ span
   the real vector space $H_n(\mathbb{C}).$ Thus it follows that 
    $$ \mathcal{D}(\lambda_1, \hdots,  \lambda_{n^2-1}) = \left\{{I_n \over n} \right\}$$ is a singleton, hence 
    $(1/n) I_n$ is an extreme density of rank $n.$ If $n^2 < k + 1,$ let us choose arbitrary $\lambda_{n^2}, \hdots, \lambda_k \in M_m(\mathbb{C})$ Hermitians 
 which are linearly independent where $m$ is large enough. 
 From Theorem 1 (iii), $(1/n)I_n \oplus 0_m$ remains extremal with respect 
 to $\lambda = (\lambda_1 \oplus 0_m, \hdots, \lambda_{n^2-1} \oplus 0_m, 0_n \oplus \lambda_{n^2}, \hdots, 0_n \oplus \lambda_k),$ hence $\mbox{Var}_{(1/n)I_n \oplus 0_m}({\bf \lambda})$ is not decomposable.
    
 Applying direct sums as above, for every large $n$ one can construct $n \times n$ extreme densities of arbitrary rank between $1$ and $\sqrt{k+1}.$  
    
   \bigskip
    
  The method we used is very similar to that of describing extreme correlations. However, the next example shows that $\mbox{Var}_D(\mathbf{X})$ is not necessarily extreme even if it is a correlation matrix and $D$ is an extreme density (with respect to some tuple).
  
  \bigskip
  
  \noindent {\bf Example 3.} Let $D$ be the projection $\mbox{diag}(1,0,\hdots, 0) \in \mathbb{R}^{n+1}.$ We define the Hermitians in $H_{n+1}(\mathbb{C})$
  $$ X_1 := \left[\begin{matrix} 
          0 & 1 \cr
          1 & 0 
      \end{matrix} \right] \oplus 0_{n-1}, \;  X_2 := \left[\begin{matrix} 
          0 & 0 & 1 \cr
          0 & 0 & 0 \cr
          1 & 0 & 0 
      \end{matrix} \right] \oplus 0_{n-2}, \; \hdots \; , $$ $$   X_n := \left[\begin{matrix} 
          0 & \hdots & 0 & 1 \cr
          \vdots & \vdots & \vdots & 0 \cr
          0 & \vdots  & \vdots  & \vdots \cr
          1 & 0 & \hdots & 0
      \end{matrix} \right].
     $$
  Then a simple calculation gives that $\mbox{Var}_D(\mathbf{X}) = I_n$ which is obviously not an extreme correlation matrix.
  
  \bigskip

  Finally, for the converse, we give an example that $\mbox{Var}_D(\mathbf{X})$ can be an extreme correlation matrix while $D$ is not necessarily extremal (with respect to $\mathbf{X}$).
  
  \bigskip
  
  \noindent {\bf Example 4.} Consider $D = (1/n) I_n \oplus 0_n \in H_{2n}(\mathbb{C}),$ $n > 2.$ Let us choose reals $x_1, \hdots, x_n$ such that $\sum_{i=1}^n x_i = 0$ and 
   $\sum_{i=1}^n nx_i^2 = 1$  hold. For any $\tilde{X}_i \in H_n(\mathbb{C}),$ $1 \leq i \leq n,$ we set 
      $$ X_i = \mbox{diag} (x_1, \hdots, x_n) \oplus \tilde{X}_i \in H_{2n}(\mathbb{C}) \qquad 1 \leq i \leq n. $$
   Then we get that $\mbox{Var}_D(\mathbf{X})$ is the $n \times n$ matrix which consists only $1$-s; that is, it is a rank-one extreme correlation matrix. From Corollary 1, $D$ cannot be extreme
   with respect to $\mathbf{X}.$


\begin{thebibliography}{99}
    
    \bibitem{B} R. Bhatia, Positive Definite Matrices, Princeton University Press, Oxford, 2007.
    
    \bibitem{BD} R. Bhatia and C. Davis, More operator versions of the Schwarz inequality, Commun. Math. Phys., {\bf 215} (2000), 239--244. 
    
    \bibitem{CV} J.P.R. Christensen and J. Vesterstr\o{}m, A note on extreme positive definite matrices, Math. Ann., {\bf 244} (1979), 65--68.
   
    \bibitem{GPW} R. Grone, S. Pierce and W. Watkins, Extremal correlation matrices, Linear Alg. and Its Appl., {\bf 134} (1990), 63--70.
    
    \bibitem{CT} C-K. Li and B-S. Tam, A note on extreme correlation matrices, SIAM J. Matrix Anal. Appl., {\bf 15 } (1994), 903--908.    
    
    \bibitem{LP} Z. L\'eka and D. Petz, Some decompositions of matrix variances, Probab. Math. Stat., {\bf 33} (2013), 191--199. 
    
    \bibitem{P} K.R. Parthasarathy, An Introduction to Quantum Stochastic Calculus, Birkh\"{a}user Verlag, Basel, 1992.  
    
\bibitem{PD} D. Petz and G. T\'oth,  Matrix variances with projections, Acta Sci. Math. (Szeged),
{\bf 78}  (2012), 683--688.

\bibitem{PD2} D. Petz and G. T\'oth,  Extremal properties of the variance and the quantum Fisher information, Phys. Review A, {\bf 87} (2013), 032324


\bibitem{SZ} H--J. Sommers and K. Zyckowski, Hilbert--Schmidt volume of the set of mixed quantum states, J. Phys. A, {\bf 36} (2003), 10115--10130.

\bibitem{Y} S. Yu, Quantum Fisher information as the convex roof of variance, preprint, arXiv:1302.5311.
  \end{thebibliography}
\end{document}